\documentclass[11pt]{article}

\usepackage{amsmath}
\usepackage{amsfonts}
\usepackage{amssymb}
\usepackage{amsthm}
\usepackage{enumitem}
\usepackage{mathtools}
\usepackage{setspace}
\usepackage{etoolbox}

\newtheoremstyle{mystyle}
{11pt}                          
{11pt}                          
{}                                      
{}                                      
{\bfseries}                     
{}                                      
{5.5pt}                         
{}                                      

\theoremstyle{mystyle}
\newtheorem{theorem}{Theorem}[section]

\newtheorem{lemma}[theorem]{Lemma}
\newtheorem{proposition}[theorem]{Proposition}

\renewenvironment{proof}[1][Proof.]{\vspace{-16.5pt} \begin{trivlist}
        \item[\hskip \labelsep {\bfseries #1}]}{\qed \end{trivlist}}

\setitemize{noitemsep, topsep=5.5pt, parsep=5.5pt, partopsep=0pt}

\allowdisplaybreaks
\appto\normalsize{
        \abovedisplayskip=5.5pt plus 2pt minus 2pt
        \belowdisplayskip=5.5pt plus 2pt minus 2pt
        \abovedisplayshortskip=5.5pt plus 2pt minus 2pt
        \belowdisplayshortskip=5.5pt plus 2pt minus 2pt}
\appto\small{
        \abovedisplayskip=5.5pt plus 2pt minus 2pt
        \belowdisplayskip=5.5pt plus 2pt minus 2pt
        \abovedisplayshortskip=5.5pt plus 2pt minus 2pt
        \belowdisplayshortskip=5.5pt plus 2pt minus 2pt}

\newcommand{\gap}{\vspace{11pt}}

\newcommand{\diag}{\operatorname{diag}}
\newcommand{\tr}{\operatorname{tr}}

\newcommand{\R}{\mathcal{R}}
\newcommand{\Rn}{\mathcal{R}^n}

\newcommand{\Sn}{\mathcal{S}^n}

\newcommand{\Hn}{\mathcal{H}^n}
\newcommand{\V}{{\cal V}}

\title{\bf  A H\"{o}lder type inequality  and an  interpolation theorem \\in
 Euclidean Jordan algebras}
\author{
        M. Seetharama Gowda\\
        Department of Mathematics and Statistics\\
        University of Maryland, Baltimore County\\
        Baltimore, Maryland 21250, USA\\
        gowda@umbc.edu
}

\date{\today}

\begin{document}

\maketitle

\begin{abstract}
In a  Euclidean Jordan algebra $\V$ of rank $n$ which carries the trace inner product, to each element 
$x$ we associate the eigenvalue vector $\lambda(x)$ whose components are the eigenvalues of $x$ written 
in the decreasing order. For any $p\in [1,\infty]$, we define the spectral $p$-norm of $x$ to be the $p$-norm of 
$\lambda(x)$ in $\R^n$.  In this paper, we show that 
$||x\circ y||_1\leq ||x||_p\,||y||_q,$
where $x\circ y$ denotes the Jordan product of two elements $x$ and $y$ in $\V$ and 
$q$ is the conjugate of $p$. 
For a linear transformation on $\V$, we state and prove an interpolation theorem relative to these spectral norms. 
In addition,   
we compute/estimate  the norms of Lyapunov transformations, quadratic representations, and  
positive transformations on $\V$.
\end{abstract}

\vspace{1cm}
\\

\newpage

\section{Introduction}
The classical H\"{o}lder and Minkowski inequalities, when stated
 in the setting of 
$\Rn$, say that for two real vectors $x=(x_1,x_2,\ldots, x_n)$ and $y=(y_1,y_2,\ldots, y_n)$ 
and for any $p\in [1,\infty]$ with conjugate $q$ (that is, $p^{-1}+q^{-1}=1$), 
$$\Big |\sum_{1}^{n}x_iy_i\Big |\leq \sum_{1}^{n}|x_iy_i|\leq ||x||_p\,||y||_q\quad\mbox{and}\quad ||x+y||_p\leq ||x||_p+||y||_p,$$
where $||x||_p$ denotes the $p$-norm of $x$, etc.
Viewing $\Rn$ as a Euclidean Jordan algebra with Jordan product
$x\circ y:=(x_1y_1,x_2y_2,\ldots, x_ny_n)$, inner product $\langle x,y\rangle=\sum_{1}^{n}x_iy_i$, and components of $x$ as eigenvalues of $x$, we may restate the above inequalities as 
$$|\langle x,y\rangle| \leq ||\lambda(x\circ y)||_1\leq ||\lambda(x)||_p\,||\lambda(y)||_q\quad\mbox{and}\quad ||\lambda(x+y)||_p\leq ||\lambda(x)||_p+||\lambda(y)||_p,$$
where $\lambda(x)$ denotes the vector of eigenvalues (here, entries) of $x$ written in the decreasing order, etc. 
Motivated by the appearance of inequalities of the above type 
in various matrix theory settings (especially for real symmetric or complex Hermitian matrices) and in the optimization literature over symmetric cones, we raise the issue of proving such  
inequalities over general Euclidean Jordan algebras. To elaborate, 
let $(\V,\circ,\langle\cdot,\cdot\rangle)$ be a Euclidean Jordan algebra of rank $n$ 
\cite{faraut-koranyi}, \cite{gowda-sznajder-tao}. We assume that $\V$ carries the trace inner product, that is,
$\langle x,y\rangle:=tr(x\circ y).$
For each $x\in \V$, we associate the eigenvalue vector
$\lambda(x)$ in $\R^n$ whose entries are the eigenvalues of $x$ written in the decreasing order. 
For $p\in [1,\infty]$, we define  the {\it spectral $p$-norm} on $\V$ by
$$||x||_{p}:=||\lambda(x)||_p,$$
where the right-hand side denotes the $p$-norm of the vector $\lambda(x)$ in $\R^n$. 
Using majorization ideas, a generalization of Thompson's triangle inequality, and case-by-case analysis 
(of five types of simple Euclidean Jordan algebras), Tao et al., \cite{tao et al} have shown that 
$||\cdot||_p$ is a norm on $\V$ thereby  establishing the Minkowski inequality in the setting of 
Euclidean Jordan algebras. 
For a comprehensive proof based on majorization and Schur-convexity theorem, see \cite{jeong-gowda}. 
Regarding the H\"{o}lder inequality, Tao et al. \cite{tao et al} 
have also shown that the inequality $|\langle x,y\rangle |\leq ||x||_p\,||y||_q$ holds for all $x$ and $y$ when $\V$ is a 
simple Euclidean Jordan algebra. For $p=2$, the inequality $||x\circ y||_1\leq ||x||_2\,||y||_2$ was proved
 in \cite{wang et al} and \cite{meng et al}.
Going beyond these special cases, in this paper we  establish the inequalities
 $$|\langle x,y\rangle|\leq ||x\circ y||_1\leq ||x||_p\,||y||_q$$
 over  general Euclidean Jordan algebras. Our related contributions include  an interpolation theorem for linear transformations on $\V$ relative to the spectral norms and computation/estimation of 
norms of Lyapunov transformations, quadratic representations, and positive transformations.

\gap

In  the first part of our paper, we establish the following  H\"{o}lder type inequality. 
To explain,  we introduce a notation and a definition. Given  any Jordan frame
$\{e_1,e_2,\ldots, e_n\}$ in $\V$, we consider the {\it ordered Jordan frame} 
${\cal E}:=\left (e_1,e_2,\ldots, e_n\right )$ and write, for any $x\in \V$,
$$\lambda(x)*\,{\cal E}:=\sum_{1}^{n}\lambda_i(x)e_i.$$
We say that two elements $x$ and $y$ in $\V$ {\it strongly operator commute} if there is an ordered Jordan frame ${\cal E}$ 
such that $x=\lambda(x)*{\cal E}$ and $y=\lambda(y)*{\cal E}.$ 

\gap

\begin{theorem}(A H\"{o}lder type inequality in Euclidean Jordan algebras) \label{main theorem}
{\it Let $x,y\in \V$ and $p\in [1,\infty]$ with conjugate $q$. Then,
\begin{equation}\label{main inequality}
||x\circ y||_1\leq ||x||_p\,||y||_q.
\end{equation}
Moreover, equality holds in {\rm (\ref{main inequality})} if and only if
\begin{itemize}
\item [$(a)$] $x$ and $y\circ \varepsilon$ strongly operator commute and
\item [$(b)$]
$\langle \lambda(x),\lambda(y\circ \varepsilon)\rangle=||\lambda(x)||_p\,||\lambda(y)||_q$ holds in $\Rn$,
\end{itemize}
where $x\circ y$ has the spectral decomposition $x\circ y=(z_1e_1+z_2e_2+\cdots+z_ke_k)-(z_{k+1}e_{k+1}+\cdots+z_ne_n)$ for some $k$, $0\leq k\leq n$ and 
$z_i\geq 0$ for all $i$,  
and $\varepsilon:=(e_1+e_2+\cdots+e_k)-(e_{k+1}+\cdots+e_n).$  
}
\end{theorem}

Our proof of the above result is based on the following generalization of the 
Fan-Theobald trace inequality of matrix theory (which is  related to von Neumann's trace inequality). The   
inequality (\ref{vn inequality}) given below extends 
the so-called rearrangement inequality of Hardy, Littlewood, and P\'{o}lya when $\V=\Rn$ \cite{marshall-olkin} and the 
Fan-Theobald trace
 inequality \cite{fan}, \cite{theobald} when $\V=\Sn$ or $\Hn$ (the algebras of $n\times n$ real/complex Hermitian matrices). For simple Euclidean Jordan algebras, this result has been observed in \cite{lim et al}, \cite{gowda-tao-cauchy}. 
Based on this simple algebra result, the rearrangement inequality, and the fact that any Euclidean Jordan algebra is a
product of simple algebras, one can prove the general result. For a different and comprehensive proof, see \cite{baes}.

\gap

\begin{theorem}(A generalized  Fan-Theobald trace inequality)
\label{generalized ft inequality}
{\it Let  $x,y\in \V$. Then,
\begin{equation}\label{vn inequality}
\langle x, y\rangle\leq  \langle \lambda(x),\lambda(y)\rangle.
\end{equation}
Moreover, equality holds in {\rm (\ref{vn inequality})} if and only if 
$x$ and $y$ strongly operator commute.
}
\end{theorem}

In the second part of the paper, we compute/estimate the (spectral) norms of  
the Lyapunov transformation $L_a$ defined by $L_a(x):=a\circ x$, the quadratic representation $P_a$ defined by 
$P_a:=2L_a^2-L_{a^2}$, and a positive (linear) transformation defined by the condition $x\geq 0\Rightarrow P(x)\geq 0$.\\ 
\\
In the final part of the paper, we describe an interpolation theorem for a linear transformation on $\V$ 
relative to the spectral norms. Based on the $K$-method of real interpolation theory \cite{lunardi}, we show that 
$$
||T||_{p\rightarrow p}\leq ||T||_{r\rightarrow r}^{1-\theta}\,\,||T||_{s\rightarrow s}^{\theta},
$$
where $T:\V\rightarrow \V$ is a linear transformation with $||T||_{p\rightarrow p}$ denoting the the norm of $T$ relative to the spectral $p$-norm and real numbers $r,s,p\in [1,\infty]$ are related by $\frac{1}{p}=\frac{1-\theta}{r}+\frac{\theta}{s}$ for some $\theta\in [0,1]$.
 
\section{Preliminaries}
The symbol $\Rn$ denotes the usual Euclidean $n$-space in which we regard elements as either row vectors or column vectors depending on the context. 
Throughout this paper, $(\V, \circ,\langle\cdot,\cdot\rangle)$ denotes a Euclidean Jordan algebra of rank 
$n$ and unit element $e$ \cite{faraut-koranyi}, \cite{gowda-sznajder-tao},
 with $x\circ y$ denoting the Jordan product and $\langle x,y\rangle$ 
denoting the inner product of $x$ and $y$ in $\V$. We specifically note that 
\begin{equation}\label{jordan property}
\langle x\circ y,z\rangle=\langle x,y\circ z\rangle\quad \mbox{for all}\,\,x,y,z\in \V.
\end{equation}
For convenience, we use the same inner product notation in $\Rn$ (which carries the usual inner product) and in $\V$.

 It is known that any Euclidean Jordan algebra is a direct product/sum 
of simple Euclidean Jordan algebras and every simple Euclidean Jordan algebra is isomorphic one of five algebras, 
three of which are the algebras of $n\times n$ real/complex/quaternion Hermitian matrices. The other two are: the algebra of $3\times 3$ octonion Hermitian matrices and the Jordan spin algebra.
We let $\Sn$ ($\Hn$) denote the algebra of all $n \times n$ real symmetric (respectively, complex Hermitian) matrices. 
 
According to the {\it spectral decomposition 
theorem} \cite{faraut-koranyi}, any element $x\in \V$ has a decomposition
$$x=x_1e_1+x_2e_2+\cdots+x_ne_n,$$
where the real numbers $x_1,x_2,\ldots, x_n$ are (called) the eigenvalues of $x$ and 
$\{e_1,e_2,\ldots, e_n\}$ is a Jordan frame in $\V$. (An element may have spectral decompositions coming from different Jordan frames, but the eigenvalues remain the same.) Then, $\lambda(x)$-- called the {\it eigenvalue vector} of $x$-- is the vector of eigenvalues of $x$ written in the decreasing order. The {\it trace  and spectral $p$-norm} of $x$ are defined by 
$$tr(x):=x_1+x_2+\cdots+x_n\quad\mbox{and}\quad ||x||_p:=||\lambda(x)||_p,$$
where $||\lambda(x)||_p$ denotes the usual $p$-norm of the  vector $\lambda(x)$ in $\Rn$. (Note that $||x||_p$ is the $p$-norm of any vector in $\Rn$ formed by
$x_1,x_2,\ldots,x_n$.)

We use the notation $x\geq 0$ ($x\leq 0$, $x>0$) when all the eigenvalues of $x$ are nonnegative (respectively, nonpositive, positive). Also, $x\geq y$ (or $y\leq x$) means that $x-y\geq 0$.
When $x\geq 0$ and has the spectral decomposition $x=\sum x_ie_i$, we define $\sqrt{x}:=\sum \sqrt{x_i}e_i$.

\gap

An element $c$ in $\V$ is an idempotent if $c^2=c$. Corresponding to such an element, the {\it Peirce decomposition of $\V$} is the orthogonal direct sum (\cite{faraut-koranyi}, page 62 and Proposition IV.1.1)
$$\V=\V(c,1)\oplus\V(c,\frac{1}{2})\oplus \V(c,0),$$
where $\V(c,\gamma):=\{x\in \V:x\circ c=\gamma\,x\}$ and $\gamma\in \{0,\frac{1}{2},1\}.$
Here $\V(c,1)$ and $\V(c,0)$ are subalgebras of $\V$ and $\V(c,1)\circ \V(c,0)=\{0\}$.
There is another related Peirce decomposition of $\V$: Corresponding to a Jordan frame $\{e_1,e_2,\ldots, e_n\}$, let
$\V_{ii}:=\V(c_i,1)=\R\,e_i$ and for $i\neq j$, $\V_{ij}:=\V(c_i,\frac{1}{2})\cap \V(c_j,\frac{1}{2})$. Then
$\V$ is the orthogonal direct sum of subspaces $\V_{ij}$ (\cite{faraut-koranyi}, Theorem IV.2.1).
Hence any element $x\in \V$ has a Peirce decomposition: $x=\sum_{i\leq j}x_{ij}$, where $x_{ij}\in \V_{ij}$.

\gap

Now, starting from the given inner product in $\V$, one can define the trace inner product $\langle x,y\rangle_{tr}:= tr(x\circ y)$ on $\V$ which is also  compatible with the given Jordan product (\cite{faraut-koranyi}, Prop. II.4.3 and Prop. III.1.5).  Various  concepts/results/decompositions  remain  the  same  when  the  given  inner  product  is
replaced by the trace inner product; in particular, for an element in $\V$, the spectral decomposition,
eigenvalues, and trace remain the same.
Under this trace inner product, the norm of any primitive element (such as an element in a Jordan frame)  
is one and so every Jordan frame becomes an orthonormal set.
{\it From now on, throughout this paper, we assume that the inner product is the\\ trace inner product, that is, 
$\langle x,y\rangle=tr(x\circ y).$}

\gap

Given a spectral decomposition $a=\sum a_ie_i$, 
we write
\begin{center}
$|a|:=\sum |a_i|e_i\quad \mbox{and}\quad ||a||_1=\sum|a_i|=tr(|a|).$
\end{center}
With this notation, we observe that 
\begin{equation}\label{ip and one-norm}
|\langle x,y\rangle|=|tr(x\circ y)|\leq tr(|x\circ y|)=||x\circ y||_1.
\end{equation}

\gap

Recall that two elements $x$ and $y$ in $\V$ {\it strongly operator commute} if there is an ordered Jordan frame ${\cal E}$ 
such that $x=\lambda(x)*{\cal E}$ and $y=\lambda(y)*{\cal E}.$
The terms `simultaneous order diagonalization' and `similar joint decomposition' have also
been used in the literature \cite{lim et al}, \cite{baes}.
Note that this notion is stronger than the usual operator commutativity where it is required that 
$x$ and $y$ have their spectral decompositions with respect to a common Jordan frame 
(or equivalently, the linear operators $L_x$ and $L_y$ commute, where $L_x(z):=x\circ z$, etc.) 
For example, in $R^2$, the vectors $x=(1,0)$ and $y=(0,1)$ operator commute, but not strongly.

\gap

Given two (column) vectors $p$ and $q$ in $\R^n$, we say that $p$ is majorized by $q$ and write $p\prec q$ if $p=Aq$ for some doubly stochastic matrix $A\in \R^{n\times n}$ \cite{marshall-olkin}. (So, $A$ is a nonnegative matrix with every 
row and column sum one. By a well-known result of Birkhoff, a doubly stochastic matrix is a convex combination of permutation matrices, see \cite{bhatia-book}.) 
For  $x,y\in \V$, we say that
$x$ is {\it majorized} by $y$ and write $x\prec y$ if $\lambda(x)\prec \lambda(y)$ in $\R^n$.
If $f:\R^n\rightarrow \R$ is a convex function and $F:=f\circ \lambda$, then (by the classical Schur-convexity theorem \cite{marshall-olkin}), we have:
\begin{equation}\label{schur convexity}
x\prec y\Rightarrow F(x)\leq F(y).
\end{equation}
See \cite{jeong-gowda} for applications of this in Euclidean Jordan algebras.

\gap

Throughout this paper, for a real number $\alpha$, we let $sgn\,\alpha$ denote $1$, $0$, or $-1$ according as
whether $\alpha$ is positive, zero, or negative.

\section{Proof of Theorem \ref{main theorem}} 

Toward establishing $(\ref{main inequality})$, we first prove a weaker inequality given below.  
It is a consequence of
 Theorem \ref{generalized ft inequality}. 

\begin{proposition} \label{prop lp inequality}
{\it 
Let  $x,y\in \V$ and $p\in [1,\infty]$ with conjugate $q$. 
Then,
\begin{equation} \label{lp inequality}
|\langle x,y\rangle |\leq ||x||_{p}\,||y||_{q}.
\end{equation}
Equality holds in {\rm (\ref{lp inequality})} if and only if, with $\eta:=sgn\,\langle x,y\rangle$, 
\begin{itemize}
\item [$(i)$] $\eta\, x$ and $y$ strongly operator commute and
\item [$(ii)$]  
$\big \langle \lambda(\eta\,x),\lambda(y)\big \rangle=||\lambda(\eta\,x)||_p\,||\lambda(y)||_q$ holds in $\Rn$.
\end{itemize}
}
\end{proposition}

\gap

\noindent The inequality (\ref{lp inequality}) for a simple algebra is noted in \cite{tao et al}, Theorem 4.2. 
A partial result for the equality in a simple algebra is stated in \cite{tao et al}, Corollary 4.2, where it is assumed that 
$x,y\geq 0$.

\gap

\begin{proof}
Without loss of generality, let $\langle x,y\rangle\neq 0$. Since $\eta:=sgn\,\langle x,y\rangle$ (which is $1$ or $-1$), from  Theorem 
\ref {generalized ft inequality} and the (classical) H\"{o}lder's inequality in $\R^n$, we have
\begin{equation} \label{inequality chain}
|\langle x, y\rangle|=\langle \eta\,x,y\rangle\leq  \langle \lambda(\eta\,x),\lambda(y)\rangle\leq ||\lambda(\eta\,x)||_p\,||\lambda(y)||_q=||x||_p\,||y||_q,
\end{equation}
where we note that 
$||\lambda(\eta\,x)||_p=||x||_p.$ This proves the inequality (\ref{lp inequality}).
Suppose $|\langle x,y\rangle |= ||x||_{p}\,||y||_{q}.$ Then, from (\ref{inequality chain}), $\langle \eta\,x,y\rangle=  \langle \lambda(\eta\,x),\lambda(y)\rangle$. It follows from Theorem \ref{generalized ft inequality}
 that $\eta\,x$ and $y$ strongly operator commute. Also, from  (\ref{inequality chain}), we get the equality  stated in $(ii)$. \\
Now suppose conditions $(i)$ and $(ii)$ hold. Then, by an application of Theorem \ref{generalized ft inequality}  
and $(ii)$ we see that 
the inequalities in (\ref{inequality chain}) turn into equalities.
\end{proof}

\gap

We now come to the proof of Theorem \ref{main theorem}. For the case of $\V=\Sn$ (or $\Hn$), the inequality
(\ref{main inequality}) can be proved using known singular values inequalities: For an $n\times n$ 
real/complex matrix $A$, let $\sigma(A):=(\sigma_1(A),\sigma_2(A),\ldots,\sigma_n(A))$ denote the 
vector of singular values of $A$ (= the eigenvalues of $\sqrt{A^*A}$) written in the decreasing order.
Then, for any two matrices $A$ and $B$, we have the inequalities
 $\sum_{1}^{n}\sigma_{i} (AB)\leq \langle \sigma(A),\sigma(B)\rangle$ and  $\sum_{1}^{n}\sigma_{i} (A+B)\leq \sum_{1}^{n}\sigma_{i}(A)+\sum_{1}^{n}\sigma_{i} (B)$ (see \cite{horn-johnson}, Theorem 3.3.14 and Corollary 3.4.3). Using these, for any $X,Y\in \Sn$ or $\Hn$, we see that
$$||X\circ Y||_1:=\sum_{1}^{n}|\lambda_{i}(X\circ Y)|=\sum_{1}^{n}\sigma_{i}(X\circ Y)=
\sum_{1}^{n}\frac{1}{2}\sigma_{i}(XY+YX)$$
$$\leq
\frac{1}{2}\Big (\sum_{1}^{n}\sigma_{i}(XY)+\sum_{1}^{n}\sigma_{i}(YX)\Big )\leq \Big \langle \sigma(X),\sigma(Y)\Big \rangle\leq ||\sigma(X)||_p\,||\sigma(Y)||_q=||\lambda(X)||_p\,||\lambda(Y)||_q.$$
The proof given below, based on majorization techniques, is comprehensive and avoids looking at particular cases of simple algebras.

\gap

\begin{proof}  
We fix $x,y\in \V$.
If  $x\circ y\geq 0$ or $x\circ y\leq 0$, then, 
$$||x\circ y||_1=\tr(|x\circ y|)=|\tr(x\circ y)|=|\langle x,y\rangle|\leq ||x||_{p}\,||y||_{p}$$
 by Proposition \ref{prop lp inequality}. Moving away from these two cases, consider the spectral decomposition of $x\circ y$ which can be written in the following form: For some  natural number $k$, $1\leq k<n$,   
$$x\circ y=(z_1e_1+z_2e_2+\cdots+z_ke_k)-(z_{k+1}e_{k+1}+\cdots+z_ne_n),$$
 where $z_i\geq 0$ for all $i$. 
Now, let $\varepsilon:=(e_1+e_2+\cdots+e_k)-(e_{k+1}+\cdots+e_n).$ Then, $\varepsilon^2=e$ and  
$|x\circ y|=(x\circ y)\circ \varepsilon$. We now claim that
\begin{equation}\label{claim}
||y\circ \varepsilon||_{q}\leq ||y||_{q}.
\end{equation}
To see this, let $c=e_1+e_2+\cdots +e_k$. By the Peirce decomposition theorem \cite{faraut-koranyi}, 
$\V$ is the orthogonal direct sum of 
$\V(c,1)$, $\V(c,\frac{1}{2})$, and $\V(c,0)$. Hence, we can write $y=u+v+w$, where $u\in \V(c,1)$, 
$v\in \V(c,\frac{1}{2})$, and  $w\in \V(c,0)$. Since $\varepsilon=2c-e$,  an 
easy computation shows that 
$$y\circ \varepsilon=u-w.$$  
As $\V(c,1)\circ \V(c,0)=\{0\}$, by  working with the spectral decompositions of $u$ in $\V(c,1)$ (which is a Euclidean Jordan algebra of rank $k$) and $w$ in $\V(c,0)$ (which is a Euclidean Jordan algebra of rank $n-k$), we see that the  
eigenvalues of $u-w$  
comprise of eigenvalues of $u$ and  $-w$; hence, $\lambda(u-w)$ is just a permutation of the vector formed by $\lambda(u)$ (which can be viewed as a vector in 
$\mathcal{R}^k$) and $-\lambda(w)$ (which can be viewed as a vector in $\mathcal{R}^{n-k}$). 
A similar statement holds for $u+w$. Hence,
$$||u-w||_q=||\lambda(u-w)||_q= ||\left [ \begin{array}{r} \lambda(u)\\-\lambda(w)\end{array}\right ]||_q=||\left [ \begin{array}{r} \lambda(u)\\\lambda(w)\end{array}\right ]||_q =||\lambda(u+w)||_q=||u+w||_q.$$
Now, it it is known (see \cite{gowda positive map}, page 52) that $$u+w\prec y.$$ With  $f(\zeta):=||\zeta||_q$ for $\zeta\in \R^n$ and 
$F(a):=f(\lambda(a))=||a||_q$ for $a\in \V$, (\ref{schur convexity}) implies 
$$||u+w||_q\leq ||y||_q.$$ It follows that 
$$||y\circ \varepsilon||_{q}=||u-w||_q=||u+w||_q\leq ||y||_q,$$
proving  (\ref{claim}). 
\\Now, 
\begin{equation}
||x\circ y||_1=\tr(|x\circ y|)=\langle |x\circ y|,e\rangle =\langle (x\circ y)\circ \varepsilon,e\rangle=\langle x,y\circ \varepsilon\rangle,
\end{equation}
where the last equality is due to (\ref{jordan property}). So, by (\ref{vn inequality})  and (\ref{claim}),
\begin{equation}\label{norm equals inner product}
||x\circ y||_1=\langle x,y\circ \varepsilon\rangle\leq \langle \lambda(x),\lambda(y\circ \varepsilon)\rangle 
\leq ||x||_{p}\,||y\circ\varepsilon||_{q}\leq ||x||_{p}\,||y||_{q}.
\end{equation}
This completes the proof of the inequality in the theorem. Now we justify the equality statement.\\
Suppose that $||x\circ y||_1= ||x||_p\,||y||_q.$ From (\ref{norm equals inner product})  
we have $\langle x,y\circ\varepsilon\rangle =\langle \lambda(x),\lambda(y\circ\varepsilon)\rangle$. From Theorem \ref{generalized ft inequality}, we get 
Item $(a)$. 
Item $(b)$ follows from (\ref{norm equals inner product}). 
Conversely, suppose conditions $(a)$ and $(b)$ hold. Then, from Theorem \ref{generalized ft inequality} and (\ref{norm equals inner product}),
$||x\circ y||_1=\langle x,y\circ \varepsilon\rangle=||x||_p\,||y||_q.$
\end{proof}

\gap

\begin{theorem}\label{some p,q norm inequalities}
{\it 
Let $p\in [1,\infty]$ with conjugate $q$. Then the following statements hold in $\V$:
\begin{itemize}
\item [$(i)$] $|\langle x,y\rangle |\leq ||x\circ y||_1\leq ||x||_p\,||y||_q$. 
\item [$(ii)$] 
$\sup_{y\neq 0}\frac{|\langle x,y\rangle|}{||y||_q}=
\sup_{y\neq 0}\frac{||x\circ y||_1}{||y||_q}=||x||_p.$
\item [(iii)] $||x\circ y||_p\leq ||x||_p\,||y||_\infty.$
\end{itemize}
}
\end{theorem} 

\gap

\begin{proof} $(i)$ follows from (\ref{main inequality}) and (\ref{ip and one-norm}).
An immediate consequence of  $(i)$ is: 
\begin{equation} \label{temp inequalities}
\sup_{y\neq 0}\frac{|\langle x,y\rangle|}{||y||_q}\leq  
\sup_{y\neq 0}\frac{||x\circ y||_1}{||y||_q}\leq ||x||_p.
\end{equation}
We now prove the reverse inequalities. Consider the spectral decomposition \\$x=\sum x_ie_i=\sum sgn (x_i)|x_i|e_i$.
First, suppose $p=\infty$. 
Then, for  $1\leq i\leq n$, 
$$|x_i|=|\langle x,e_i\rangle |=\frac{|\langle x,e_i\rangle |}{||e_i||_1}\leq \sup_{y\neq 0}\frac{|\langle x,y\rangle|}{||y||_1},$$ hence $||x||_\infty\leq \sup_{y\neq 0}\frac{|\langle x,y\rangle|}{||y||_1}$. By (\ref{temp inequalities}), the reverse inequality also holds. Thus, $(ii)$ holds when $p=\infty$.\\
Now, let $u:=\sum sgn (x_i)|x_i|^{\frac{p}{q}}e_i$ when $1<p<\infty$ and 
$u:=\sum sgn (x_i)\,e_i$ when $p=1$. We easily verify that 
$|\langle x,u\rangle|=||x\circ u||_1=||x||_p\,||u||_q.$ Thus, the inequalities in (\ref{temp inequalities}) turn into equalities, proving $(ii)$ for $1\leq p<\infty$. 
\\
Now, 
$$||x\circ y||_p=\sup_{z\neq 0}\frac{|\langle x\circ y,z\rangle|}{||z||_q}=\sup_{z\neq 0}\frac{|\langle x\circ z,y\rangle|}{||z||_q}\leq \sup_{z\neq 0}\frac{||x\circ z||_1\,||y||_\infty}{||z||_q}\leq ||x||_p\,||y||_\infty,$$
where the first equality comes from $(ii)$, the first inequality comes from an application of (\ref{lp inequality}), and the second inequality comes from $(ii)$.
This proves $(iii)$.
\end{proof}

\gap

\noindent{\bf Remark 1.} The above result shows that the spectral norms $||\cdot||_p$ and $||\cdot||_q$ on $\V$ are dual to each other.  We also have the following inequalities:
$$||x\circ y||_1\leq ||x||_1\,||y||_\infty\quad \mbox{and}\quad 
||x\circ y||_\infty\leq ||x||_\infty\,||y||_\infty.$$
The first inequality has been observed in \cite{tao-yuan} in a simple algebra setting based on a case-by-case analysis.

\gap

\section{Pointwise inequalities for positive transformations}
Given $a\in \V$, we define the corresponding {\it Lyapunov transformation} $L_a$ and 
{\it quadratic representation} $P_a$ on $\V$ by
$$L_a(x):=a\circ x\quad\mbox{and}\quad P_a(x)=2a\circ (a\circ x)-a^2\circ x\quad (x\in \V).$$
Now, expressed in terms of $L_a$, Theorem \ref{main theorem} says that 
$$||L_a(x)||_1\leq ||x||_p\,||a||_q$$ for all $a,x\in \V$ and $p\in [1,\infty]$ with conjugate $q$.
In this section, we consider such inequalities for quadratic representations and more generally for the so-called positive
transformations. Recall that a linear transformation $P:\V\rightarrow \V$ is said to be a {\it positive transformation} \cite{gowda positive map} 
if
$$x\geq 0\Rightarrow P(x)\geq 0.$$ Writing $P^*$ for the adjoint/transpose
 of a linear transformation $P$, we note that if $P$ is positive, then $P^*$ is also positive as $\langle P^*(z),y\rangle =\langle z,P(y)\rangle \geq 0$ for all $y,z\geq 0$.

\gap

Here are some examples of positive transformations:
\begin{itemize}
\item [$\bullet$] Any nonnegative matrix on the algebra $\Rn$.
\item [$\bullet$] Any quadratic representation $P_a$ on $\V$ \cite{faraut-koranyi}.
\item [$\bullet$] For any $A\in \R^{n\times n}$, the transformation $P$ defined on $\Sn$  by $P(X):=AXA^T$.
\item [$\bullet$] $P=L^{-1}$ on $\V$, where $L:\V\rightarrow \V$ is linear, positive stable (which means that all eigenvalues of $L$ have positive real parts) and satisfies the $Z$-property \cite{gowda-tao-z}:
$$x\geq 0,y\geq 0,\,\langle x,y\rangle =0\Rightarrow \langle L(x),y\rangle \leq 0.$$ Specifically,
\begin{itemize}
\item [$(i)$] On the algebra $\Rn$, $P=A^{-1}$, where $A$ is a positive stable $Z$-matrix (meaning that its off-diagonal entries are nonpositive);
\item [$(ii)$] On the algebra $\Hn$, 
$P=L_{A}^{-1}$, where $A$ is a complex $n\times n$  positive stable matrix and  $L_A(X):=AX+XA^*$. 
The transformation $L_A$ (also called a Lyapunov transformation) appears in dynamical systems.
\end{itemize}
\item [$\bullet$] $P$ is a {\it doubly stochastic transformation} on $\V$ \cite{gowda positive map}. This means that $P$ is positive\\ and $P(e)=e=P^*(e)$. Being a generalization of a doubly stochastic matrix, such a  transformation has the following property (\cite{gowda positive map}, Theorem 6):
$$x=P(y)\Rightarrow x\prec y.$$
\item [$\bullet$] {\it `Schur product'} induced transformation $P$ defined  as follows (\cite{gowda-tao-sznajder}, Proposition 2.2): Fix a positive semidefinite matrix $A=[a_{ij}]\in \Sn$ and a Jordan frame 
 $\{e_1,e_2,\ldots, e_n\}$ in $\V$. Corresponding to this Jordan frame, we write the Peirce decomposition (\cite{faraut-koranyi}, Theorem IV.2.1) of any $x\in \V$:
$x=\sum_{i\leq j}x_{ij}$. Then,  
$$P(x):=A\bullet x=\sum_{i\leq j}a_{ij}x_{ij}.$$
It is known (\cite{gowda positive map}, Example 8) that if such an $A$ has all ones on its diagonal, then $P$ is doubly stochastic.
\end{itemize}

\gap

The following result gives pointwise estimates for positive transformations.

\begin{theorem}\label{theorem for positive transformations}
{\it 
Let $P$ be a positive transformation on $\V$. For any $x\in \V$ and $p\in [1,\infty]$ with conjugate $q$, we have 
\begin{itemize}
\item [$(a)$] 
$||P(x)||_1  \leq ||x||_p\,||P^*(e)||_q.$ In particular, $||P(x)||_1  \leq ||x||_1\,||P^*(e)||_\infty.$
\item [$(b)$] $||P(x)||_p\leq ||x||_\infty\,||P(e)||_p.$ In particular, $||P(x)||_\infty\leq ||x||_\infty\,||P(e)||_\infty.$
\end{itemize}
}
\end{theorem}

\gap

\begin{proof}
$(a)$ We start with the observation that when $u\leq v$ and $-u\leq v$ in $\V$, we have $||u||_1\leq ||v||_1$. This is easy to see:
Writing the spectral decomposition $u=\sum u_ie_i$, we have $u_i=\langle u,e_i\rangle \leq \langle v,e_i\rangle$ and 
similarly, $-u_i\leq \langle v,e_i\rangle$; thus, $|u_i|\leq \langle v, e_i\rangle$ for all $i$ and so,
$$||u||_1=tr(|u|)\leq \sum \langle v,e_i\rangle =\langle v,e\rangle =tr(v)=||v||_1$$ as $v\geq 0$. Now, for any $x\in \V$, $x\leq |x|$ and $-x\leq |x|$; hence using the positivity of $P$,
$P(x)\leq P(|x|)$ and $-P(x)\leq P(|x|)$ and so
$$||P(x)||_1\leq ||P(|x|)||_1=\Big \langle P(|x|),e\Big\rangle=\Big\langle |x|,P^*(e)\Big\rangle\leq \Big\langle \lambda(|x|),\lambda(P^*(e))\Big\rangle
\leq ||x||_p\,||P^*(e)||_q,$$
where the second inequality comes from (\ref{vn inequality}) and the last inequality is just the classical H\"{o}lder's inequality.
\\$(b)$ Since $P$ is a positive transformation, $P^*$ is also positive. 
Hence, by applying $(a)$ to $P^*$ and $y$, we get $||P^*(y)||_1\leq ||y||_q\,||P(e)||_p.$ Now, 
by an application of (\ref{lp inequality}), we get 
$$||P(x)||_p=\sup_{y\neq 0}\frac{|\langle P(x),y\rangle|}{||y||_q}=\sup_{y\neq 0}\frac{|\langle x, P^*(y)\rangle|}{||y||_q}\leq \sup_{y\neq 0}\frac{||x||_\infty\,||P^*(y)||_1}{||y||_q}\leq ||x||_\infty\,||P(e)||_p.$$
\end{proof}

\gap

Here are some illustrations of the above theorem.

\begin{itemize}
\item [$\bullet$] Let $A\in \R^{n \times n}$ and consider the positive transformation $P$ on $\Sn$ defined by $P(X):=AXA^T$. Then, for any $X\in \Sn$, $P^*(X):=A^TXA$. So, with $e=I$ (the identity matrix), we have the inequalities 
$$||AXA^T||_1\leq ||X||_p\,||A^TA||_q\quad\mbox{and}\quad ||AXA^T||_p\leq ||X||_\infty\,||AA^T||_p.$$
\item [$\bullet$] Let $A=[a_{ij}]\in \Sn$ be positive semidefinite. Then, considering the `Schur product' positive transformation $X\mapsto A\bullet X$, we have, for any $X\in \Sn$,
$$\rho(A\bullet X)=||A\bullet X||_\infty\leq ||X||_\infty\,||A\bullet I||_\infty =||X||_\infty\,(\max_{1\leq i\leq n}|a_{ii}|)=\rho(X)\,\rho(\diag(A)),$$ 
where $\rho(X)$ denotes the spectral radius of $X$ and $\diag(A):=A\bullet I$ with $I$ denoting the identity matrix.
We remark that eigenvalue and spectral radius inequalities for the Schur/Hadamard product have been well-studied in the matrix theory literature.
\end{itemize}

\gap

We now specialize the above result to $P_a$. It is well-known 
that $P_a$ is self-adjoint and positive. Moreover, $P_a(e)=a^2$. 
Hence,  we have the following:
For any $a,x\in \V$ and $p\in [1,\infty]$ with conjugate $q$, 
$$
||P_a(x)||_1\leq  ||x||_p\,||a^2||_q\quad \mbox{and}\quad ||P_a(x)||_p\leq ||x||_\infty\,||a^2||_p.
$$

As we see below, some finer inequalities can be obtained. 

\begin{theorem}\label{majorization between quadratic and Lyapunov}
{\it 
For  any $a,x\in \V$,
\begin{equation}\label{majorization between pa and la}
P_a(x)\prec a^2\circ x.
\end{equation}
Hence, for any  $p\in [1,\infty]$ with conjugate $q$,
\begin{equation} \label{inequality for pa}
||P_a(x)||_1\leq  ||a^2\circ x||_1\leq ||x||_p\,||a^2||_q\quad \mbox{and}\quad ||P_a(x)||_p\leq ||a^2\circ x||_p\leq ||x||_\infty\,||a^2||_p.
\end{equation}
}
\end{theorem}

\gap

\begin{proof}
The inequalities in (\ref{inequality for pa}) follow from (\ref{majorization between pa and la}) 
 by an application of 
(\ref{schur convexity}) (with
$f$ denoting the usual $p$-norm on $\Rn$) and Theorem \ref{some p,q norm inequalities}.
We now prove (\ref{majorization between pa and la}). Since $P_a(x)\prec a^2\circ x$ means that 
$\lambda(P_a(x))\prec \lambda(a^2\circ x)$, by continuity of the eigenvalue map $\lambda$ and the compactness of the set of all $n\times n$ doubly stochastic matrices, it is enough to prove (\ref{majorization between pa and la}) when $a$ is invertible (that is, all eigenvalues of $a$ are nonzero). {\it So, assume that $a$ is invertible.} 
Then, we have the formula (see \cite{schmieta-alizadeh}, Lemma 8, Item 3)
$$P_{a,a^{-1}}P_a=L_{a^2},$$
where $P_{a,a^{-1}}:=L_aL_{a^{-1}}+L_{a^{-1}}L_a-L_{a\,\circ\, a^{-1}}.$
By  Lemma \ref{lemma in appendix} in the Appendix, the linear transformation $P_{a,a^{-1}}$  is invertible and 
its inverse, $(P_{a,a^{-1}})^{-1}$, is doubly stochastic. Now, writing  
$$P_a(x)=(P_{a,a^{-1}})^{-1}(a^2\circ x)$$
and invoking Theorem 6 in \cite{gowda positive map}, we see that  $P_a(x)\prec a^2\circ x.$ 

\end{proof} 
 
We now mention some consequences of the majorization inequality (\ref{majorization between pa and la}).
\begin{itemize}
\item [(1)] Writing $\lambda_{\max}(u)$ and $\lambda_{\min}(u)$ for the maximum and minimum of eigenvalues of $u$, we have
$$\lambda_{\max}(P_a(x))\leq \lambda_{\max}(a^2\circ x)\quad \mbox{and}\quad \lambda_{\min}(P_a(x))\geq \lambda_{\min}(a^2\circ x).$$ 
\item [(2)] When $a\geq 0$, we have $P_{\sqrt{a}}(x)\prec a\circ x.$ Such a majorization inequality was proved 
in \cite{meng et al} on a case-by-case basis under the assumptions that $\V$ is simple, $a>0$, $x>0$, and $a\circ x>0$.
\item [(3)] When $a\geq 0$, for any real number $\mu$, $P_{\sqrt{a}}(x)-\mu e\prec a\circ x-\mu e.$ It follows from (\ref{schur convexity}) that for any $p\in [1,\infty]$,
$$||P_{\sqrt{a}}(x)-\mu e||_p\leq ||a\circ x-\mu e||_p.$$
Such inequalities, for $p\in \{2,\infty\}$, appear in interior point methods, see e.g., \cite{schmieta-alizadeh}, Lemma 30.
\end{itemize}

\section{Norms of  Lyapunov transformations, quadratic representations, and positive transformations}
In a recent paper \cite{gowda positive map}, it was shown that for a positive transformation $P$ on a Euclidean Jordan algebra, 
the infinity norm of $P$ is attained at the unit element. This result applies  to the quadratic transformation $P_a$ and the inverse of a positive stable $Z$-transformation \cite{gowda-tao-z} on a Euclidean Jordan algebra. In this section, we consider calculating the norms of the Lyapunov transformation $L_a$, 
the quadratic representation $P_a$, and positive transformations relative to spectral norms. We remark that questions related to the norm of the Lyapunov 
transformation $L_A$ on $\Sn$ (defined by $L_A(X):=AX+XA^T$ for $A\in \R^{n\times n}$) arise in connection with stability of dynamical systems, see  
 \cite{feng et al} and citations therein. 

\gap

Given a linear transformation $T:\V\rightarrow \V$, and $r,s\in [1,\infty]$, we define the norm of the operator 
$T:\,(\V,||\cdot||_r)\rightarrow (\V,||\cdot||_s)$ by
$$||T||_{r\rightarrow s}:=\sup_{x\neq 0}\frac{||T(x)||_{s}}{||x||_r}.$$
By the duality of norms (see Theorem \ref{some p,q norm inequalities}), we immediately see that 
$$||T||_{r\rightarrow s}=||T^*||_{s^\prime\rightarrow r^\prime},$$
where $T^*$ denotes the adjoint/transpose of $T$ and $r^\prime$ ($s^\prime$) denotes the conjugate of $r$ (respectively, of $s$). 
\gap

Now consider the spectral decomposition 
$a=\sum a_ie_i$. Then,
$$L_a(e_i)=a_ie_i\quad\mbox{and}\quad P_a(e_i)=a_i^2e_i$$
for all $i$. For any $r,s\in [1,\infty]$, $||e_i||_r=||e_i||_s=1$ and so,
$|a_i|=||a_ie_i||_s=||a\circ e_i||_s\leq ||L_a||_{r\rightarrow s}\,||e_i||_r=||L_a||_{r\rightarrow s}$. Taking the maximum over $i$, we see that
\begin{equation}\label{lower bound for norm of la}
||a||_\infty\leq ||L_a||_{r\rightarrow s}\quad (r,s\in [1,\infty]).
\end{equation}
Similarly, 
\begin{equation}\label{lower bound for norm of pa}
||a^2||_\infty=||a||_\infty^2\leq ||P_a||_{r\rightarrow s}\quad (r,s\in [1,\infty]).
\end{equation}

\gap

\begin{theorem}\label{la norms}
{\it 
For any $a\in \V$ and $p\in [1,\infty]$ with conjugate $q$, the following statements hold:
\begin{itemize}
\item [$(i)$] $||L_a||_{\infty\rightarrow q}=||L_a||_{p\rightarrow 1}=||a||_q.$
\item [$(ii)$] $||L_a||_{1\rightarrow q}=||L_a||_{p\rightarrow \infty}=||a||_\infty.$
\item [$(iii)$] $||L_a||_{p\rightarrow p}=||a||_\infty.$
\end{itemize}
}
\end{theorem}

\begin{proof} 
$(i)$ As $L_a$ is self-adjoint, the first equality comes from the
 duality of norms. The second equality is  immediate from Item $(ii)$ in 
 Theorem \ref{some p,q norm inequalities}. 
\\
$(ii)$ The first equality is due to the duality of norms. Now for the second equality. We have, from  Remark 1,  
 $||a\circ x||_\infty \leq ||a||_\infty\,||x||_\infty$. As  $||x||_\infty\leq ||x||_p$, 
we see that  $||a\circ x||_\infty \leq ||a||_\infty\,||x||_p$ 
and so, $||L_a||_{p\rightarrow \infty}\leq ||a||_\infty.$
On the other hand, 
$||a||_\infty\leq ||L_a||_{p\rightarrow \infty}$ from (\ref{lower bound for norm of la}). This proves the equality 
$||L_a||_{p\rightarrow \infty}=||a||_\infty.$
\\
$(iii)$ From Theorem \ref{some p,q norm inequalities}$(iii)$,
$||a\circ x||_p\leq ||x||_p\,||a||_\infty.$
From this, we get $||L_a||_{p\rightarrow p}\leq ||a||_\infty.$
On the other hand, 
$||a||_\infty\leq ||L_a||_{p\rightarrow p}$ from (\ref{lower bound for norm of la}). Thus,  
$||a||_\infty= ||L_a||_{p\rightarrow p}$.
\end{proof}

\gap
\begin{theorem}
{\it 
Let $P$ be a positive transformation on $\V$ and $p\in [1,\infty]$ with conjugate $q$. Then,
\begin{itemize}
\item [$(i)$] $||P||_{\infty\rightarrow p}= ||P(e)||_p$ and $||P||_{p\rightarrow 1}= ||P^*(e)||_q$.
\item [$(ii)$] $||P||_{p\rightarrow \infty}\leq ||P(e)||_{\infty}$ and $||P||_{1\rightarrow p}\leq ||P^*(e)||_{\infty}$.
\end{itemize}
}
\end{theorem}

\begin{proof}
$(i)$ From Theorem  \ref{theorem for positive transformations}(b), $||P(x)||_p\leq ||x||_\infty\,||P(e)||_p$ with equality when $x=e$. Hence $||P||_{\infty\rightarrow p}=\sup_{x\neq 0}\frac{||P(x)||_p}{||x||_\infty}= ||P(e)||_p$. 
The dual version of this gives the second statement in $(i)$.\\
$(ii)$  
 From Theorem  \ref{theorem for positive transformations}(b), we have $||P(x)||_\infty\leq ||x||_\infty\,||P(e)||_\infty\leq ||x||_p\,||P(e)||_\infty.$ This gives 
$||P||_{p\rightarrow \infty}\leq ||P(e)||_{\infty}$. The second statement is the dual version of this.\\
\end{proof}

\noindent{\bf Remark 2.} The above result shows that for a positive transformation on $\V$, 
$||P||_{1\rightarrow 1}\leq ||P^*(e)||_\infty$ and $||P||_{\infty \rightarrow \infty}\leq ||P(e)||_\infty$. 
Using Theorem \ref{interpolation theorem} (see the next section), for any $p\in [1,\infty]$ we have
$$||P||_{p\rightarrow p}\leq ||P^*(e)||_\infty^{\frac{1}{p}}\,||P(e)||_\infty^{1-\frac{1}{p}}$$
and when $P$ is self-adjoint, $||P||_{p\rightarrow p}\leq ||P(e)||_\infty.$ 
To see a special case, suppose $L:\V\rightarrow \V$ is linear, positive stable, and satisfies the $Z$-property (see Section 3 for definitions). Then, 
$$||L^{-1}||_{p\rightarrow p}\leq ||(L^*)^{-1}(e)||_\infty^{\frac{1}{p}}\,||L^{-1}(e)||_\infty^{1-\frac{1}{p}}.$$
In particular, by taking  $\V=\Hn$ and $L=L_A$ with $A$ positive stable (see Section 3), we can estimate $||L_A^{-1}||_{p\rightarrow p}.$ See \cite{bhatia} for a discussion of this type of  an estimate on the space of all $n\times n$ 
complex matrices.

\gap

For quadratic representations, we can compute the norms precisely. 

\begin{theorem}
For any $a\in \V$ and $p\in [1,\infty]$ with conjugate $q$, the following statements hold:
\begin{itemize}
\item [$(i)$] $||P_a||_{p\rightarrow 1}=||P_a||_{\infty\rightarrow q}= ||a^2||_q.$
\item [$(ii)$] $||P_a||_{1\rightarrow p}=||P_a||_{q\rightarrow \infty}=||a||_\infty^2.$
\item [$(iii)$] $||P_a||_{p\rightarrow p}=||a||_\infty^2.$
\end{itemize}
\end{theorem}

\begin{proof}
$(i)$ Since $P_a$ is self-adjoint and $P_a(e)=a^2$, this comes from the 
previous theorem, Item $(i)$.
\\
$(ii)$ From Item $(ii)$ in the previous theorem, $||P_a||_{1\rightarrow p}=||P_a||_{q\rightarrow \infty}\leq ||a||_\infty^2.$ The reverse inequality follows from (\ref{lower bound for norm of pa}).
\\
As a consequence of Item $(ii)$,
$||P_a||_{1\rightarrow 1}=||a||_\infty^2=||P_a||_{\infty\rightarrow \infty}$. By invoking Theorem \ref{interpolation theorem} (see the next section),  we see that  $||P_a||_{p\rightarrow p}\leq ||a||_\infty^2$. Since the reverse inequality also holds, see  (\ref{lower bound for norm of pa}), we have
$||P_a||_{p\rightarrow p}= ||a||_\infty^2.$
\end{proof}

\section{An interpolation theorem}
In this section, we prove the following  interpolation theorem for a linear transformation on $\V$  with respect to the spectral norms.

\begin{theorem}\label{interpolation theorem}
{\it 
Suppose $1\leq r,s,p\leq \infty$,  $0\leq \theta\leq 1$, and 
\begin{equation}\label{r,s,p}
\frac{1}{p}=\frac{1-\theta}{r}+\frac{\theta}{s}.
\end{equation}
Then, for  any linear transformation $T:\V\rightarrow \V$ we have
\begin{equation}\label{interpolation inequality}
||T||_{p\rightarrow p}\leq ||T||_{r\rightarrow r}^{1-\theta}\,\,||T||_{s\rightarrow s}^{\theta}.
\end{equation}
In particular,
$$||T||_{p\rightarrow p}\leq ||T||_{1\rightarrow 1}^{\frac{1}{p}}\,\,||T||_{\infty\rightarrow \infty}^{1-\frac{1}{p}}.
$$
}
\end{theorem}

There are numerous interpolation theorems in analysis, two classical
 ones being the Riesz-Thorin and Marcinkiewicz 
interpolation theorems.  The interpolation theorems are usually proved using either the real or the complex methods. 
We present a proof of the above theorem using the $K$-method of real interpolation theory \cite{lunardi}. 
While the above result deals with the norm of $T$ relative to the same spectral norm (such as $||T||_{p\rightarrow p}$),
 we  anticipate a broader result similar to the Riesz-Thorin theorem that deals with the norm of $T$ relative to two spectral norms (such as $||T||_{p_0\rightarrow p_1}$). 
We note that a Riesz-Thorin type result is available for linear transformations
 on the space of complex $n\times n$ matrices with respect to Schatten $p$-norms, see the interpolation theorem of Calder\'{o}n-Lions (\cite{reed-simon}, Theorem IX.20). A key idea in our proof is the use of a majorization result that connects a 
$K$-functional defined on $\V$ to a $K$-functional on an $L_p$-space.
\\

Before presenting the proof, we describe some background material.
Corresponding to our Euclidean Jordan algebra $\V$ of rank $n$, we let $\Omega:=\{1,2,\ldots, n\}$
and  $\mu$ denote the  measure on (the power set of) $\Omega$ with $\mu(\{k\})=1$ for all $k\in \Omega$. Let  $L_{p}(\Omega)$ (abbreviated as $L_p$) denote the corresponding Lebesgue measure space (consisting, for our consideration,
 only of real valued functions).
We regard any element $f$ in $L_{p}(\Omega)$ either as an $n$-tuple or as a real valued function on $\Omega$. We let $||f||_p$ denote the usual $p$-norm of $f$.  
{\it We assume the notation/conditions of Theorem \ref{interpolation theorem}.}
We will use the following abbreviations:
$$V_r:=(V,||\cdot||_r),\,\,V_s:=(V,||\cdot||_s),\,\,M_r:=||T||_{r\rightarrow r},\,\,\mbox{and}\,\,M_s:=||T||_{s\rightarrow s}.$$ 
For any real number  $t>0$, $x\in \V$, and $f:\Omega\rightarrow \R$, we consider all possible decompositions $x=a+b$  with $a,b\in \V$ and $f=g+h$ with $g,h:\Omega\rightarrow \R$, and define the $K$-functionals:
$$K(t,x,V_r,V_s):=\inf\Big \{||a||_r+t\,||b||_s:x=a+b\Big \},$$
and 
$$K(t,f,L_r,L_s):=\inf\Big \{||g||_r+t\,||h||_s:f=g+h\Big \}.$$

We describe/recall some preliminary results.

\begin{proposition} (\cite{lunardi}, Definition 1.2 and Example 1.27)\label{equality in lp}
{\it 
Suppose $1\leq r, s,p\leq \infty$, $r< s$,  $0< \theta< 1$, and
$$\frac{1}{p}=\frac{1-\theta}{r}+\frac{\theta}{s}.
$$
Then, for any $f:\Omega\rightarrow \R$,
$$||f||_p=\Big [\int_{0}^{\infty} \Big ( t^{-\theta}K(t,f,L_r,L_s)\Big )^p\frac{dt}{t}\Big ]^{\frac{1}{p}}.$$
}
\end{proposition}

\gap

The following can be regarded as a majorization result. In a simple Euclidean Jordan algebra, it is known that $\lambda(a+b)\prec \lambda(a)+\lambda(b)$ \cite{gowda-tao-cauchy} so that $\lambda(a+b)=A(\lambda(a)+\lambda(b))$ for some doubly stochastic matrix $A$. For a general Euclidean Jordan algebra, we have the following.

\begin{proposition} (\cite{jeong-gowda}, Proposition 8) \label{majorization}
{\it Given $a,b\in \V$, there exist doubly stochastic matrices $A$ and $B$ in $\R^{n\times n}$ such that 
$$\lambda(a+b)=A\lambda(a)+B\lambda(b).$$
}
\end{proposition}

\gap

Based on the above majorization result, we connect the two $K$-functionals defined earlier.
 
\begin{lemma}\label{equality of two Ks}
{\it For any $t>0$ and $x\in \V$, we have
$$K(t,x,V_r,V_s)=K(t,\lambda(x), L_r,L_s).$$
}
\end{lemma}

\begin{proof}
We fix $x\in \V$ and consider the decomposition $x=a+b$. By Proposition \ref{majorization},
$$\lambda(x)=A\lambda(a)+B\lambda(b),$$
where $A$ and $B$ are doubly stochastic matrices in $\R^{n\times n}$.
Let $g:=A\lambda(a)$ and $h:=B\lambda(b)$ so that 
$$\lambda(x)=g+h.$$
As $A$ and $B$ are convex combinations of permutation matrices (by Birkhoff's Theorem \cite{bhatia-book}), we see that $||g||_r\leq ||\lambda(a)||_r=||a||_r$ and $||h||_s\leq ||\lambda(b)||_s=||b||_s.$ Hence, for any $t>0$,
$$K(t,\lambda(x),L_r,L_s)\leq ||g||_r+t\,||h||_s\leq ||a||_r+t\,||b||_s.$$
As this holds for any decomposition $x=a+b$, taking the infimum, 
$$K(t,\lambda(x), L_r,L_s)\leq K(t,x,V_r,V_s).$$
Now for the reverse inequality. Consider any decomposition $\lambda(x)=g+h$ where $g,h:\Omega\rightarrow \R.$ Corresponding to the spectral decomposition $x=\sum \lambda_i(x)e_i$, we define 
$$a:=\sum g(i)e_i\quad\mbox{and}\quad b:=\sum h(i)e_i.$$
Then, $x=a+b$ in $\V$. So,
$$K(t,x,V_r,V_s)\leq ||a||_r +t\,||b||_s=||g||_r+t\,||h||_s.$$
Taking the infimum, we get
$$K(t,x,V_r,V_s)\leq K(t,\lambda(x), L_r,L_s).$$
This completes the proof of the lemma.
\end{proof}

\gap

\begin{lemma} \label{inequality of Ks}
{\it 
Let $T:\V\rightarrow \V$ be linear and nonzero. Then,
$$K(t,T(x),V_r,V_s)\leq M_r\,K\Big (\frac{M_s}{M_r}t,x,V_r,V_s\Big ).$$
}
\end{lemma}

\begin{proof} Fix $x\in \V$ with decomposition $x=a+b$. Then, $T(x)=T(a)+T(b)$ and so,
$$K(t,T(x),V_r,V_s)\leq ||T(a)||_r+t\,||T(b)||_s\leq M_r||a||_r+t\,M_s||b||_s\leq M_r\Big [||a||_r+t\frac{M_s}{M_r}||b||_s\Big ].$$
Taking the infimum over all decompositions $x=a+b$, we get the stated inequality.
\end{proof}

We now come to the proof of Theorem \ref{interpolation theorem}.

\gap
 
\begin{proof} As (\ref{interpolation inequality}) holds when $T=0$ or $r=s$ or when $\theta\in \{0,1\}$, we assume that 
 $T\neq 0$, $r\neq s$, and $0<\theta<1$. We first assume that $r<s$ (so that conditions of Proposition \ref{equality in lp} are met).\\   
We fix $x\in \V$ and let $y:=T(x)$. Then, 
$||y||_p=||\lambda(y)||_p$ and 
$$\begin{array}{rcl}
||\lambda(y)||_p^p & = & \int_{0}^{\infty} \Big [t^{-\theta}K(t,\lambda(y),L_r,L_s)\Big ]^p\frac{dt}{t}\\
          &=& \int_{0}^{\infty} \Big [ t^{-\theta}K(t,y,V_r,V_s)\Big ]^p\frac{dt}{t}\\
          &\leq & \int_{0}^{\infty} \Big [ t^{-\theta}M_r\,K\Big(\frac{M_s}{M_r}t,x,V_r,V_s\Big )\Big ]^{p}\frac{dt}{t}\\
          &=& \int_{0}^{\infty}\Big [ \Big (\frac{M_r}{M_s}t\Big )^{-\theta}M_r K(t,x,V_r,V_s)\Big ]^p
\frac{dt}{t}\\
&=& \Big (\frac{M_r}{M_s}\Big )^{-\theta\,p}M_r^{p}\int_{0}^\infty \Big [t^{-\theta}K(t,x,V_r,V_s)\Big ]^p\frac{dt}{t}\\
&=& M_r^{p(1-\theta)}M_s^{p\theta}\,||x||_p^{p},
\end{array} 
$$
where the first equality is due to Proposition \ref{equality in lp},  the second equality is due to Lemma 
\ref{equality of two Ks}, the first inequality is due to Lemma \ref{inequality of Ks}, and the third equality is due to a change of variable.
Also, the last equality is seen by applying Lemma \ref{equality of two Ks} and  Proposition \ref{equality in lp}. 
Hence,
$$||T(x)||_p\leq M_r^{1-\theta}M_s^{\theta}\,||x||_p.$$
This implies that 
$$||T||_{p\rightarrow p}\leq ||T||_{r\rightarrow r}^{1-\theta}\,||T||_{s\rightarrow s}^{\theta}.$$
Note that we proved this inequality  under the assumption that $r<s$. When $s<r$, we let $\phi:=1-\theta$ and observe that $\frac{1}{p}=\frac{1-\phi}{s}+\frac{\phi}{r}$. Then, by what has been proved, 
$$||T||_{p\rightarrow p}\leq ||T||_{s\rightarrow s}^{1-\phi}\,||T||_{r\rightarrow r}^{\phi}=||T||_{r\rightarrow r}^{1-\theta}\,||T||_{s\rightarrow s}^{\theta}.$$
We thus have (\ref{interpolation inequality}) in all cases. In particular, by putting 
$r=1$ and $s=\infty$ we  get 
$$||T||_{p\rightarrow p}\leq ||T||_{1\rightarrow 1}^{\frac{1}{p}}\,||T||_{\infty\rightarrow \infty}^{1-\frac{1}{p}}.$$
\end{proof}
\section{Appendix}

For $a,b\in \V$, we define (see \cite{schmieta-alizadeh}, page 4 or \cite{faraut-koranyi}, page 32)
$$P_{a,b}:=L_aL_b+L_bL_a-L_{a\,\circ\, b}.$$ 

\begin{lemma}\label{lemma in appendix}
Suppose $a\in \V$ is invertible. Then, the linear transformation $P_{a,a^{-1}}$ is invertible and 
$(P_{a,a^{-1}})^{-1}$
is doubly stochastic.
\end{lemma}

\begin{proof}
We consider the spectral decomposition $a=\sum a_ie_i$, where the eigenvalues $a_i$ are nonzero and $\{e_1,e_2,\ldots, e_n\}$ is a Jordan frame. With respect to this Jordan frame, we consider the Peirce decomposition of any element $u\in \V$ (\cite{faraut-koranyi}, Theorem IV.2.1) in the form $u=\sum_{i\leq j}u_{ij}.$ Then, from \cite{gowda-tao-sznajder} (page 720), 
$$L_a(u)=\sum_{i\leq j}\frac{a_i+a_j}{2}\,u_{ij}\quad\mbox{and}\quad L_{a^{-1}}(u)=\sum_{i\leq j}\frac{a_i^{-1}+a_j^{-1}}{2}\,u_{ij}.$$
Since $a\circ a^{-1}=e$, an  easy computation shows that 
$$P_{a,a^{-1}}(u)=\sum_{i\leq j} \frac{a_i^2+a_j^2}{2a_{i}a_{j}}\,u_{ij}.$$
As $P_{a,a^{-1}}(u)=0\Rightarrow u=0$, the linear transformation $P_{a,a^{-1}}$ is invertible and 
$$(P_{a,a^{-1}})^{-1}(u)=\sum_{i\leq j} \frac{2a_{i}a_{j}}{a_i^2+a_j^2}\,u_{ij}.$$ 
Now consider the real symmetric matrix $A=[a_{ij}]$, where $a_{ij}=\frac{2a_{i}a_{j}}{a_i^2+a_j^2}$.
By considering the functions $\phi_i(t)$ in $L_2([0,\infty))$, $i=1,2,\ldots, n$, defined by 
$$\phi_i(t)=\sqrt{2}\,a_i\,e^{-a_i^2\,t},$$
we see that the inner product (computed in $L_2([0,\infty))$)
$$\langle \phi_i,\phi_j\rangle=\int_{0}^{\infty}2a_ia_je^{-(a_i^2+a_j^2)t}\,dt=a_{ij}.$$
Hence, $A$ is the Gram matrix corresponding to the set $\{\phi_1,\phi_2,\ldots,\phi_n\}$ in  $L_2([0,\infty))$. It follows that $A$ is positive semidefinite. Now, using the definition of `Schur product' induced transformation (see Section 4) 
$$(P_{a,a^{-1}})^{-1}(u)= A\bullet u\,\,(u\in \V).$$
As $A$ is positive semidefinite and has ones on the diagonal, from Example 8 in \cite{gowda positive map} we see 
that the transformation $u\rightarrow A\bullet u$ is doubly stochastic. This proves that $(P_{a,a^{-1}})^{-1}$
is doubly stochastic.
\end{proof}
 


\begin{thebibliography}{}

\bibitem{baes} M. Baes, {\it Convexity and differentiability properties of spectral functions and spectral mappings on 
Euclidean Jordan algebras,} Linear Alg. Appl. 422 (2007) 664-700. 

\bibitem{bhatia} R. Bhatia, \emph{A note on the Lyapunov equation},  Linear Alg.  Appl., 259 (1997) 71-76.

\bibitem{bhatia-book} R. Bhatia, \emph{Matrix Analysis}, Graduate Texts in Mathematics, Springer-Verlag, New York, 1997.

\bibitem{fan} K. Fan, \emph{On a theorem of Weyl concerning eigenvalues of linear transformations I}, 
Proc. Nat. Acad. Sci. U.S.A., 35 (1949) 652-655.

\bibitem{faraut-koranyi} 
J. Faraut and A. Kor\'{a}nyi, \emph{Analysis on Symmetric Cones}, Oxford University Press, Oxford, 1994.

\bibitem{feng et al} J. Feng, J. Lam, G. Yang, and Z. Li, \emph{On a conjecture about the norm of Lyapunov mappings}, Linear Alg. Appl., 465 (2015) 88-103.

\bibitem{gowda positive map} M.S. Gowda, \emph{Positive and doubly stochastic maps, and majorization in 
Euclidean Jordan algebras}, Linear Alg.  Appl., 528 (2017) 40-61. 

\bibitem{gowda-sznajder-tao} M.S. Gowda, R. Sznajder, and J. Tao, \emph{Some P-properties for linear 
transformations on Euclidean Jordan algebras}, Linear Alg.  Appl.,  393 (2004) 203-232. 

\bibitem{gowda-tao-z}
M.S. Gowda and J.Tao, \emph{Z-transformations on proper and symmetric cones,} Math. Program., Series B, 117 (2009) 195-222.

\bibitem{gowda-tao-cauchy} 
M.S. Gowda and J. Tao, \emph{The Cauchy interlacing theorem in simple Euclidean Jordan algebras and some consequences,} 
Linear and Multi. Alg., 59 (2011) 65-86.

\bibitem{gowda-tao-sznajder} M.S. Gowda, J. Tao, and R. Sznajder,  \emph{Complementarity properties of Peirce-diagonalizable linear transformations on Euclidean Jordan algebras,}  Optim. Methods Soft., 27 (2012) 719-733.
 
\bibitem{horn-johnson} R. Horn and C.R. Johnson, \emph{Topics in Matrix Analysis}, Cambridge University Press, New York, 
1991.

\bibitem{jeong-gowda}
J. Jeong and M.S. Gowda, \emph{Spectral sets and functions in Euclidean Jordan algebras}, Linear Alg.  Appl.,  518 (2017) 31-56. 

\bibitem{lim et al}
Y. Lim, J. Kim, and L. Faybusovich,
\emph{Simultaneous diagonalization on simple Euclidean Jordan algebras and its applications}, Forum Math., 15 (2003) 639-644.

\bibitem{lunardi} A. Lunardi, \emph{Interpolation Theory}, Third Edition,
Edizioni della Normale, Pisa, 2018.

\bibitem{marshall-olkin}
A.W. Marshall, I. Olkin, and B.C. Arnold, \emph{Inequalities: Theory of majorization and its Applications}, Springer (2010).

\bibitem{meng et al}
Q. Meng, J. Tao, G. Wang, and X. Chi, \emph{Some norm inequalities in Euclidean Jordan algebras}, 
Pacific Jour. Opt., 13 (2017) 315-324. 

\bibitem{reed-simon} M. Reed and B. Simon, \emph{Fourier Analysis, Self-adjointness} (Methods of Modern Mathematical Physics, Vol. 2), Academic Press, 1975.

\bibitem{schmieta-alizadeh} S.H. Schmieta and F. Alizadeh, \emph{Extension of primal-dual interior point algorithms to symmettric cones}, Math. Prog. A 96 (2003) 409-438.

\bibitem{tao et al}
J. Tao, L. Kong, Z. Luo, and N. Xiu, \emph{Some majorization inequalities in Euclidean Jordan algebras}, Linear Alg. Appl.,
161 (2014) 92-122.

\bibitem{tao-yuan} J. Tao and X. Yuan, \emph{A characterization of minimal elements in simple Euclidean Jordan algebras}, Private communication. 

\bibitem{theobald} C.M. Theobald, \emph{An inequality for the trace of the product of two symmetric matrices,} Math. Proc. Cambridge Phil. Soc., 77 (1975) 265.

\bibitem{wang et al}
G. Wang, J. Tao, and L. Kong, \emph{A note on an inequality involving Jordan product in Euclidean Jordan algebras}, Optim. Lett., 10 (2016) 731-736.
\end{thebibliography}
\end{document}